\documentclass[12pt]{amsart}

\usepackage{amssymb}

\usepackage{enumerate, mathrsfs}
\usepackage[all]{xy}
\usepackage{url}

\usepackage{graphicx}

\makeatletter
\@namedef{subjclassname@2010}{%
  \textup{2010} Mathematics Subject Classification}
\makeatother

\newtheorem{thm}{Theorem}[section]
\newtheorem{cor}[thm]{Corollary}

\newtheorem{prop}[thm]{Proposition}

\newtheorem{question}{Question}

\theoremstyle{definition}

\newtheorem{rem}[thm]{Remark}

\newcommand{\R}{\mathbb{R}}
\newcommand{\Q}{\mathbb{Q}}
\newcommand{\scrC}{\mathscr{C}}
\newcommand{\scrfC}{\tilde{\mathscr{C}}}

\newcommand{\I}{I^d_{n,j}}
\newcommand{\E}{E^d_{n,j}}
\newcommand{\bE}{\overline{E}{}}

\newcommand{\fE}{\mathrm{f}E}
\newcommand{\fI}{\mathrm{f}I}
\newcommand{\Top}{\mathrm{Top}}
\newcommand{\id}{\mathrm{id}}

\newcommand{\zero}{\boldsymbol 0}

\numberwithin{equation}{section}

\begin{document}


\baselineskip=17pt



\title{Deloopings of the spaces of long embeddings}

\author[K.~Sakai]{Keiichi Sakai}
\address{Department of Mathematical Sciences, Shinshu University, 3-1-1 Asahi, Matsumoto, Nagano 390-8621, Japan}
\email{ksakai@math.shinshu-u.ac.jp}

\date{}

\begin{abstract}
The homotopy fiber of the inclusion from the long embedding space to the long immersion space is known to be an iterated based loop space (if the codimension is greater than two).
In this paper we deloop the homotopy fiber to obtain the topological Stiefel manifold, combining results of Lashof and of Lees.
We also give a delooping of the long embedding space, which can be regarded as a version of Morlet-Burghelea-Lashof's delooping of the diffeomorphism group of the disk relative to the boundary.
As a corollary, we show that the homotopy fiber is weakly equivalent to a space on which the framed little disks operad acts possibly nontrivially, and hence its rational homology is a (higher) BV-algebra in a stable range of dimensions.
\end{abstract}



\maketitle

\section{Introduction}
Let $E^d=\E$ (resp.\ $I^d=\I$) be the space of {\em long $j$-embeddings} (resp.\ {\em long $j$-immersions}) in $\R^n$, that is, smooth embeddings $f:\R^j\hookrightarrow\R^n$ (resp.\ immersions $\R^j\looparrowright\R^n$) 
such that $f(x)=(x,{\boldsymbol 0})$ if $|x|\ge 1$.
Here ``$d$'' indicates that we are considering differentiable maps.
We also consider the space $\fE^d_{n,j}$ ($\fI^d_{n,j}$) of {\em framed} long embeddings (immersions) $\R^j\times (-\epsilon ,\epsilon )^{n-j}\to\R^n$.
Budney \cite{Budney03} defined an action of little $(j+1)$-disks operad $\scrC_{j+1}$ on (a space equivalent to) $\fE^d_{n,j}$.
Consequently $\fE^d_{n,j}$ ($n-j\ge 3$) turns out to be weakly equivalent to a $(j+1)$-fold loop space by the loop space recognition principle \cite{May271}.
Budney's $\scrC_{j+1}$-action also applies to $\fI^d_{n,j}$ in such a way that the inclusion $\fE^d_{n,j}\to\fI^d_{n,j}$ is a map of $\scrC_{j+1}$-spaces.
Thus the space $\bE^d_{n,j}$, the homotopy fiber of $\fE^d_{n,j}\to\fI^d_{n,j}$ (or equivalently of $E^d_{n,j}\to I^d_{n,j}$), is also a $\scrC_{j+1}$-space and hence a $(j+1)$-fold loop space if $n-j\ge 3$ (this argument is the same as the proof of \cite[Proposition~1.1]{Turchin10}).
Sinha \cite{Sinha04} also proved that $\bE^d_{n,1}$ ($n\ge 4$) is weakly equivalent to a double loop space, using a cosimplicial method.
Based on Sinha's work, Salvatore \cite{Salvatore06} showed that $E^d_{n,1}$ ($n\ge 4$) is weakly equivalent to a double loop space with double loop maps $E^d_{n,1}\to\fE^d_{n,1}$ and $E^d_{n,1}\to\bE^d_{n,1}$.

A natural question is; what is the delooping of $\bE^d_{n,j}$ (and of $\fE^d_{n,j}$)?
Dwyer-Hess \cite{DwyerHess10} and Tourtchine \cite{Tourtchine10} independently described a delooping of $\bE^d_{n,1}$ ($n\ge 4$) as the derived space of maps between some operads.
The purpose of this paper is to give a simple delooping of $\bE^d_{n,j}$ which had already appeared implicitly in Lashof's paper \cite{Lashof76}.

\begin{thm}\label{thm:main}
If $n-j\ge 3$ and $n\ge 5$, then $\bE^d_{n,j}$ is weakly equivalent to the $(j+1)$-fold based loop space of the topological Stiefel manifold $V^t_{n,j}$.
\end{thm}

The {\em topological Stiefel manifold} $V^t_{n,j}$ (``$t$'' suggests that it consists of topological maps) is defined to be the orbit space $\Top (n)/\Top (n,j)$, where $\Top (n,j)$ is the topological group of germs at $\zero$ of homeomorphisms $\R^n\xrightarrow{\approx}\R^n$ which restrict to the identity on $\R^j\times\{ 0\}^{n-j}$, and $\Top (n):=\Top (n,0)$.
Though $V^t_{n,j}$ is not a manifold in the usual sense, we follow the classical terminology.

Since the orthogonal group acts in a nontrivial way on the topological Stiefel manifold, a possibly nontrivial {\em BV-structure} on $H_*(\bE^d_{n,j})$ is deduced by \cite[Example~2.5]{SalvatoreWahl03} in a range of dimensions.

\begin{cor}\label{cor:BV}
If $n-j\ge 3$, $n\ge 5$ and $n\ge 2j+1$, then $\bE^d_{n,j}$ is weakly equivalent to a space on which the framed $(j+1)$-disks operad $\scrfC_{j+1}$ acts in a nontrivial way.
Consequently $H_*(\bE^d_{n,j};\Q )$ is a $\text{BV}_{j+1}$-algebra \cite[Definition~5.2]{SalvatoreWahl03}.
\end{cor}

It is well known, though not so frequently mentioned, that $E^d_{n,j}$ ($n-j\ge 3$) is weakly equivalent to a $j$-fold loop space, because $\pi_0E^d_{n,j}$ is a group if $n-j\ge 3$ \cite{Haefliger66} and $\scrC_j$ acts on $E^d_{n,j}$ in a similar fashion to the case of $j$-fold based loop spaces.
We can also describe a delooping of $E^d_{n,j}$.

\begin{prop}\label{prop:deloop_unframed}
If $n-j\ge 3$ and $n\ge 5$, then $E^d_{n,j}$ is weakly equivalent to $\Omega^jV_{n,j}^{t/d}$, where $V^{t/d}_{n,j}$ is the homotopy fiber of the natural inclusion from the (usual) Stiefel manifold $V^d_{n,j}=O(n)/O(n-j)$ to $V^t_{n,j}$.
\end{prop}

The delooping in Proposition~\ref{prop:deloop_unframed} can be seen as a ``positive codimension version'' of Morlet-Burghelea-Lashof's delooping of the diffeomorphism group ${\rm Diff}(D^n,\partial )=E^d_{n,n}$ of the disk relative to the boundary \cite{BurgheleaLashof74-1, Morlet69};
\begin{equation}\label{eq:MBL}
 {\rm Diff}(D^n,\partial )\sim\Omega^{n+1}(\Top(n)/O(n)).
\end{equation}
Indeed \eqref{eq:MBL} can be written as $E^d_{n,n}\sim\Omega^nV^{t/d}_{n,n}$, since $\Top (n)=V^t_{n,n}$, $O(n)=V^d_{n,n}$, and $O(n)\to\Top (n)\to\Top (n)/O(n)$ is a fiber bundle \cite[Theorem~4.1]{Gleason50} and hence a Serre fibration.

The proof of the following is similar to that of Corollary~\ref{cor:BV}.

\begin{cor}\label{cor:BV_unframed}
If $n-j\ge 3$ and $n\ge 5$, then $E^d_{n,j}$ is weakly equivalent to a space on which $\scrfC_j$ acts in a nontrivial way.
Consequently $H_*(E^d_{n,j};\Q )$ is a $\text{BV}_j$-algebra. 
\end{cor}

Proposition~\ref{prop:deloop_unframed} gives rise to an alternative proof of the useful fact which was proved in \cite[Proposition~3.9~(1)]{Budney08} by means of a spinning method (in a wider range of dimensions).
In fact the isomorphism in Corollary~\ref{cor:spinning} below coincides with that given in \cite{Budney08}.

\begin{cor}[\cite{Budney08}]\label{cor:spinning}
If $n-j\ge 3$ and $n\ge 5$, then $\pi_kE^d_{n,j}\cong\pi_0E^d_{n+k,j+k}$ for $k\le 2(n-j)-5$.
\end{cor}

Here we mention some possible advantages of our delooping of $\bE^d$.

First, we might be able to describe the (co)homology of $\bE^d$ in terms of that of $V^t$, and possibly BV-algebra structure from Corollaries~\ref{cor:BV} and \ref{cor:BV_unframed} might produce new homology classes of $\bE^d$ (see also \S\ref{s:Q}, Question~\ref{q3}).
The author indeed proved in \cite{K08, K09} that the Browder operation induced by Budney's $\scrC_2$-action \cite{Budney03} yields a nontrivial homology class of $\fE^d_{n,1}$ for odd $n\ge 3$ (see \cite{Longoni04} for a similar result).
In fact this homology class can also be obtained by using the BV-operator introduced in \cite{K10} arising from Hatcher's cycle \cite{Hatcher99}.
It would be an interesting question to determine, using $H_*(V^t)$, the generating set of $H_*(\bE^d)$ as a BV-algebra.

Second, the proof of Corollary~\ref{cor:spinning} does not require the celebrated ``Goodwillie calculus'' as in \cite{Budney08}.
Instead we need the knowledge of the homotopy groups of $V^t_{n,j}$ \cite{Lashof76, Millett75}.
So far many interesting results on the (homology of) embedding spaces have been obtained by means of Goodwillie calculus (see for example \cite{AroneLambrechtsVolic07, LambrechtsTurchinVolic10} and the papers already referred above).
Perhaps we might be able to give alternative proofs for some of these results using $V^t_{n,j}$ as in Corollary~\ref{cor:spinning}, and if this is the case, it would be curious to compare these two methods.

In \S\ref{s:proof} we prove the above results.
In \S\ref{s:Q} some related questions are listed.

\section*{Acknowledgments}
The author deeply feels grateful to Masamichi Takase for communicating the presence of Lashof's paper and for teaching him the essential points of it, and to Victor Turchin for his careful reading of the draft of this paper and for many valuable comments.
The author also thanks Ryan Budney and Tadayuki Watanabe for fruitful discussions.
The author is partially supported by JSPS KAKENHI Grant numbers 23840015, 25800038.

\section{Proofs}\label{s:proof}
Let $E^t=E^t_{n,j}$ and $I^t=I^t_{n,j}$ be the spaces of locally flat topological long embeddings and immersions $\R^j\to\R^n$ respectively.
Let $E^{t/d}=E^{t/d}_{n,j}$ and $I^{t/d}=I^{t/d}_{n,j}$ be the homotopy fibers of the inclusions $E^d\to E^t$ and $I^d\to I^t$ respectively.

\begin{thm}[{\cite[Theorem A ({\it t/d})]{Lashof76}}]\label{thm:Lashof}
If $n-j\ge 3$ and $n\ge 5$, then the map $E^{t/d}_{n,j}\to I^{t/d}_{n,j}$ is a weak homotopy equivalence.
\end{thm}

\begin{rem}
Theorem~\ref{thm:Lashof} was stated in \cite{Lashof76} in terms of simplicial sets.
As mentioned in \cite[Appendix]{Lashof76}, by a work of \v{C}ernavski\u{\i} \cite{Cernavskii69}, the simplicial sets of locally flat topological embeddings or immersions used in \cite{Lashof76} are homotopy equivalent to the singular complexes of our space $E^t_{n,j}$ or $I^t_{n,j}$ if the conditions on $n$ and $j$ are satisfied.
Therefore we always assume $n-j\ge 3$ and $n\ge 5$ throughout this paper.
\end{rem}

\begin{proof}[Proof of Theorem~\ref{thm:main}]
Consider the following commutative diagram consisting of six fibration sequences;
\begin{equation}\label{eq:diagram}
\begin{split}
 \xymatrix{
  \bE^{t/d}\ar[r]\ar[d]          & \bE^d\ar[r]\ar[d] & \bE^t\ar[d] & \\
  E^{t/d}\ar[r]\ar[d]^-{\rm (a)} & E^d\ar[r]\ar[d]   & E^t\ar[d]\ar[r]^-{\simeq} & {\ast} \\
  I^{t/d}\ar[r]                  & I^d\ar[r]         & I^t &
 }
\end{split}
\end{equation}
where $\bE^*$ denotes the homotopy fiber of $E^*\to I^*$, $*=d$, $t$, $t/d$.
Since by Theorem~\ref{thm:Lashof} the map (a) in \eqref{eq:diagram} is a weak equivalence, $\bE^{t/d}$ is weakly contractible and hence $\bE^d\to\bE^t$ is a weak equivalence.

On the other hand, since $E^t$ is contractible by the Alexander trick (\cite[p.~146, Example]{Lashof76}), $\Omega I^t\to\bE^t$ is a homotopy equivalence.
Theorem~\ref{thm:main} follows from Lees' topological Smale-Hirsch theorem $I^t_{n,j}\xrightarrow{\sim}\Omega^jV^t_{n,j}$ \cite{Lees69}.
\end{proof}

\begin{rem}
In fact Lees' theorem \cite{Lees69} asserts that there exists a weak equivalence from the space $\fI^t_{n,j}$ of topological framed long immersions $\R^j\times (-\epsilon ,\epsilon )^{n-j}\looparrowright\R^n$ to $\Omega^j\Top (n)$, which equivalence fits into the following diagram of fibration sequences;
\[
 \xymatrix{
  \Omega^j\Top (n,j)\ar[r]\ar@{=}[d] & \fI^t_{n,j}\ar[r]\ar[d]_-{\sim}^-{\text{Lees}} & I^t_{n,j}\ar[d] \\
  \Omega^j\Top (n,j)\ar[r]           & \Omega^j\Top (n)\ar[r]                         & \Omega^jV^t_{n,j}
 }
\]
Thus we have $I^t_{n,j}\xrightarrow{\sim}\Omega^jV^t_{n,j}$ (on the component containing the base point).
\end{rem}

\begin{rem}\label{rem:t=pl}
The above proof works even if the spaces of topological maps are replaced by those of piecewise-linear (PL) maps.
In this case the proof relies on Haefliger-Poenaru's theorem $I^{PL}_{n,j}\xrightarrow{\sim}\Omega^jV^{PL}_{n,j}$ \cite{HaefligerPoenaru64} (the space of locally flat PL immersions and the PL Stiefel manifold $V^{PL}$ are defined analogously to the topological case).
If $n-j\ge 3$, then $V^{PL}_{n,j}\to V^t_{n,j}$ is a homotopy equivalence by \cite[Proposition~({\it t/pl})]{Lashof76}, and hence $V^{PL}$ may replace $V^t$.
\end{rem}

\begin{proof}[Proof of Proposition~\ref{prop:deloop_unframed}]
As noted in \cite[p.~146, Example]{Lashof76}, there is a fibration sequence
\[
 E^d\to I^d\to I^t.
\]
This is because we can deduce a weak equivalence $E^d\xrightarrow{\sim}I^{t/d}$ which makes \eqref{eq:diagram} homotopy commutative, using Theorem~\ref{thm:Lashof} and the fact that $E^t$ is contractible.

The weak equivalences $I^*_{n,j}\to\Omega^jV^*_{n,j}$ ($*=d,t$) are both given as taking the germs of the long immersions.
Thus $I^d_{n,j}\to I^t_{n,j}$ is equivalent to the $j$-fold loop map of $V^d_{n,j}\to V^t_{n,j}$ and hence the fiber $E^d$ is equivalent to $\Omega^jV^{t/d}_{n,j}$.
\end{proof}

\begin{proof}[Proof of Corollary~\ref{cor:BV}]
As explained in \cite[Proposition~2.3, Example~2.5]{SalvatoreWahl03}, the $\scrC_{j+1}$-action on a $(j+1)$-fold based loop space $\Omega^{j+1}X$ can be extended to that of $\scrfC_{j+1}$ by using a basepoint-preserving action of $SO(j+1)$ on $X$ ($\Omega^{j+1}X$ is then acted on by $SO(j+1)$ by the conjugation).
Thus we need a (nontrivial) action of $SO(j+1)$ on $V^t_{n,j}$ which preserves the basepoint; the orbit of $\id_{\R^n}$.
One of the easiest choices is the restriction of the conjugation of $SO(n-j)=\id_{\R^j}\oplus SO(n-j)$ on $\Top (n)$ (see Remark~\ref{rem:action} below), which descends to a basepoint-preserving action on $V^t_{n,j}$ since it preserves $\Top (n,j)$.
If $n-j\ge j+1$, this action restricts to that of $SO(j+1)$.

The $\text{BV}_{j+1}$-structure on $H_*(\Omega^{j+1}V^t_{n,j};\Q )$ is a consequence of the $\scrfC_{j+1}$-action \cite[Theorem~5.4, Example~5.5]{SalvatoreWahl03}.
The weak equivalence $\bE^d_{n,j}\xrightarrow{\sim}\Omega^{j+1}V^t_{n,j}$ implies $H_*(\bE^d_{n,j};\Q )\cong H_*(\Omega^{j+1}V^t_{n,j};\Q )$ and this completes the proof.
\end{proof}

\begin{rem}\label{rem:action}
One of the reasons why we chose the conjugation in the proof of Corollary~\ref{cor:BV} is that it seems meaningful from the viewpoint of immersions;
for example, we can define an action $SO(j)\times V^*_{n,j}\to V^*_{n,j}$ ($*=d,t$) as the conjugation $(g,f)\mapsto (g\oplus\id_{\R^{n-j}})\circ f\circ g^{-1}$, here $V^*_{n,j}$ is regarded as the space of germs at $\zero$ of (smooth or topological) embeddings $(\R^j,\zero )\hookrightarrow (\R^n,\zero )$.
Under the Smale-Hirsch/Lees equivalence $I^*_{n,j}\xrightarrow{\sim}\Omega^jV^*_{n,j}$, the induced action of $SO(j)$ on $\Omega^j V^*_{n,j}$ by conjugation corresponds to the natural conjugation action of $SO(j)$ on long immersions $I^*_{n,j}$.
This action seems meaningful since it would produce new immersions via ``spinning'', and can be used for the proof of Corollary~\ref{cor:BV_unframed}.
However in the proof of Corollary~\ref{cor:BV} we adopted $SO(n-j)$ instead of $SO(j)$ as the space acting on $V_{n,j}^*$, because the action explained here unfortunately does not extend to that of $SO(j+1)$. 
But the action in Corollary~\ref{cor:BV} may also be meaningful, because it ``rotates'' embeddings in the orthonormal direction $\{\zero\}^j\times\R^{n-j}$ and looks similar to the ``Gramain cycle'' \cite{Gramain77, Budney10}.
At present the author does not know whether the action given in the proof of Corollary~\ref{cor:BV} yields a nontrivial BV-operation on $H_*(\bE^d_{n,j})$, nor whether there are other significant actions.
\end{rem}

\begin{proof}[Proof of Corollary~\ref{cor:BV_unframed}]
The proof is similar to that of Corollary~\ref{cor:BV}, but in this case, as explained in Remark~\ref{rem:action}, we may use the conjugation $SO(j)\times V^*_{n,j}\to V^*_{n,j}$ given by $(g,f)\mapsto (g\oplus\id_{\R^{n-j}})\circ f\circ g^{-1}$ ($*=d,t$) which preserves the basepoints $[\id_{\R^n}]\in V^*_{n,j}$ (this action requires no dimension assumptions).
Since $V^d_{n,j}\to V^t_{n,j}$ is $SO(j)$-equivariant under the conjugation, $SO(j)$ also acts on the homotopy fiber $V^{t/d}_{n,j}$ preserving the basepoint, the constant path at $[\id_{\R^n}]\in V^t_{n,j}$.
The weak equivalence $E^d_{n,j}\xrightarrow{\sim}\Omega^jV^{t/d}_{n,j}$ induces an isomorphism on homology.
\end{proof}

\begin{proof}[Proof of Corollary~\ref{cor:spinning}]
Using Haefliger-Millett's theorem \cite{Millett75}, Lashof proved in \cite[Proposition~({\it t/d})]{Lashof76} that, when $n-j\ge 3$ and $m\le 2n-j-5$, there is an isomorphism
\begin{equation}\label{eq:HaefligerMillett}
 \pi_m(V^{t/d}_{n,j})\cong\pi_{m+1}(G,O,G_{n-j})
\end{equation}
where $G_q$ is the space of degree one maps $S^{q-1}\to S^{q-1}$ and $G$ is its stable suspension (see \cite{Haefliger66}).
On the other hand, by Haefliger's classification theorem \cite{Haefliger66},
\begin{equation}\label{eq:Haefliger}
 \pi_{m+1}(G,O,G_q)\cong\pi_0E^d_{m+q,m}
\end{equation}
for $q\ge 3$.
Proposition~\ref{prop:deloop_unframed} and the isomorphisms \eqref{eq:HaefligerMillett} and \eqref{eq:Haefliger} deduce
\[
 \pi_{m-j}E^d_{n,j}\stackrel{\text{Prop }\ref{prop:deloop_unframed}}{\cong}
 \pi_mV^{t/d}_{n,j}\stackrel{\eqref{eq:HaefligerMillett}}{\cong}
 \pi_{m+1}(G,O,G_{n-j})\stackrel{\eqref{eq:Haefliger}}{\cong}\pi_0E^d_{n-j+m,m}
\]
for $n-j\ge 3$ and $j\le m\le 2n-j-5$.
Putting $k=m-j$ completes the proof.
\end{proof}

\section{Questions}\label{s:Q}
In \cite{Budney03, BudneyCohen05} the space $E^d_{3,1}$ is proved to be a free $\scrC_2$-object, and hence $H_*(E^d_{3,1};\Q )$ is a free Poisson algebra \cite{Cohen533}.
An analogous result for $\bE^d_{n,j}$ for general $n,j$ ($n-j\ge 3$) would be derived if the answer of the following question is affirmative.

\begin{question}\label{q1}
Is $V^t_{n,j}$ a $(j+1)$-fold suspension?
\end{question}

Salvatore proved in \cite{Salvatore06} that $E^d_{n,1}$ ($n\ge 4$) is weakly equivalent to a double loop space.
The following question asks whether the similar result holds for general $n,j$.

\begin{question}\label{q2}
Is $V^{t/d}_{n,j}$ a based loop space with any $(j+1)$-fold loop map $\Omega^jV^{t/d}_{n,j}\to\Omega^{j+1}V^t_{n,j}$?
\end{question}

\begin{question}\label{q3}
How do the BV-structures of Corollary~\ref{cor:BV} and of \cite{K10} relate to each other?
Do they produce any new operation other than cycles of Gramain \cite{Gramain77} and Hatcher \cite{Hatcher99}?
\end{question}

\providecommand{\bysame}{\leavevmode\hbox to3em{\hrulefill}\thinspace}
\providecommand{\MR}{\relax\ifhmode\unskip\space\fi MR }
\providecommand{\MRhref}[2]{%
  \href{http://www.ams.org/mathscinet-getitem?mr=#1}{#2}
}
\providecommand{\href}[2]{#2}
\end{document}